\documentclass[11pt,twoside]{amsart}
\usepackage[english]{babel}
\usepackage{amssymb}
\usepackage{amsthm}
\usepackage{latexsym}
\usepackage{calligra}

\newcommand{\bl}{\left|}                            
\newcommand{\br}{\right|}                           
\newcommand{\be}{\begin{equation}}
\newcommand{\ee}{\end{equation}}

\def\c{\cite }
\def\d{{\,\rm d\,}}

\def\phi{\varphi}
\newcommand{\bi}{\begin{itemize}}
\newcommand{\ei}{\end{itemize}}
\newcommand{\bn}{\begin{enumerate}}
\newcommand{\en}{\end{enumerate}}
\def\R{\mathbb{R}}                                     
\def\N{\mathbb{N}}

\def\cF{\mathcal{F}}
\def\cU{\mathcal{U}}

\def\fF{\mathfrak{F}}

\newcommand{\hide}[1]{}          

\theoremstyle{plain}
\newtheorem{thm}{Theorem}[section]

\newtheorem{prop}[thm]{Proposition}

\newtheorem{lemma}[thm]{Lemma}
\newtheorem{rem}[thm]{Remark}

\newtheorem{example}[thm]{Example}

\theoremstyle{definition}
\newtheorem{definition}[thm]{Definition}

\numberwithin{equation}{section}

\begin{document}
 \title{Finite elements in some vector lattices of nonlinear operators}
 \author{M.~A.~Pliev and M.~R.~Weber}
 \address{\footnote{\,The first named author was supported by the Russian Foundation of Fundamental Research,
 the grant number 14-01-91339.} South Mathematical Institute of the Russian Academy of Sciences \\
 str. Markusa 22 \\
 362027 Vladikavkaz,  Russia}
 \email{maratpliev@gmail.com}
 \address{Technical University Dresden \\
 Department of Mathematics, Institut of Analysis \\ 
 01062 Dresden, Germany}
 \email{martin.weber@tu-dresden.de}


\begin{abstract}
We study the collection of finite elements $\Phi_{1}(\mathcal{U}(E,F))$ in the vector lattice $\mathcal{U}(E,F)$ of 
orthogonally additive, order bounded (called abstract Uryson) operators between two vector lattices $E$ and $F$, 
where $F$ is Dedekind complete. 
In particular, for an atomic vector lattice $E$ it is proved that for a finite element in $\phi\in \mathcal{U}(E,\R)$ 
there is only a finite set of mutually disjoint atoms, where $\phi$ does not vanish and, 
for an atomless vector lattice the zero-vector is the only finite element in the band of $\sigma$-laterally continuous 
abstract Uryson functionals.     
We also describe the ideal $\Phi_{1}(\cU(\R^n,\R^m))$ for $n,m\in\N$ and consider rank one operators  
to be finite elements in $\cU(E,F)$.
\end{abstract}

\maketitle

%
%
\section{Introduction}
%
%
The last time finite elements in vector lattices have been an object of an active investigation
\cite{Ch-1, Ch-2, Ch-3, H, Mal, W}.
This class of elements in Archimedean vector lattices was introduced as an abstract analogon of continuous
functions (on a topological space) with compact support by Makarov and Weber in  in 1972, see \cite{MW74}.
Recently a systematic treatment of finite elements in vector lattices appeared in \cite{Web14}.
On the other hand the study of nonlinear maps between vector lattices is also a growing area of Functional analysis,
where the background has to be found in the nonlinear integral operators, see e.g. \cite{KZPS}.
The interesting class of nonlinear, order bounded, orthogonally additive operators, called abstract Uryson operators,
was introduced and studied in 1990 by Maz\'{o}n and Segura de Le\'{o}n \cite{Maz-1,Maz-2}, and then
considered to be defined on lattice-normed spaces by Kusraev and Pliev in \cite{Ku-1,Ku-2,Pl-3}.
Now a theory of abstract Uryson operators is also a subject of intensive investigations \cite{BAP,Gum,PP}.
In this paper the investigation of finite
elements is extended to the vector lattice $\mathcal{U}(E,F)$ of
abstract Uryson operators from a vector lattice $E$ to a Dedekind complete vector lattice $F$.
\section{Preliminaries}
%
%
The  goal of this section is to introduce some basic definitions and facts. General information on vector lattices
the reader can find in the books \cite{Al,Ku, Web14, Za}.
\par
\begin{definition} \label{def:ddmjf0}
Let $E$ be an Archimedean vector lattice.
An element $\varphi\in E$ called {\it finite}, if there is an element $z\in E$ satisfying the following condition:
for any element $x\in E$ there exists a number $c_{x}>0$ such that the following inequality holds
\[
|x|\wedge n|\varphi|\leq c_{x}z\,\,\,\text{for all}\,\,n\in\N.
\]
\end{definition}
 For a finite element $\varphi$ the (positive) element $z$ is called a {\it majorant} of $\varphi$.
 If a finite element $\varphi$ possesses a majorant which itself is a finite element then $\varphi$
 is called {\it totally finite}.  The collections of all finite and totally finite elements of a vector lattice $E$
are ideals in $E$ and will be denoted by $\Phi_1(E)$ and $\Phi_2(E)$, respectively.
It is clear that $0$ is always a finite element. For our purpose we mention that the relations
$\Phi_1(E)=E$ and $\Phi_1(E)=\{0\}$ are possible (for the complete list of the relations between
$E, \, \Phi_1(E)$ and $\Phi_2(E)$ see \c{Web14}, section 6.2).
Finite elements in vector and Banach lattices have been studied in \cite{Ch-1},
in sublattices of vector lattices in \cite{Ch-2}, in $f$-algebras and product algebras in \c{Mal} and in vector lattices
of linear operators in \cite{Ch-3, H}.
\par
The relations between the finite elements in $E$ and the finite elements in vector sublattices of $E$ are manifold.
The result we need later is the following.
\begin{prop}[\c{Web14}, Theorem 3.28 and Corollary 3.29]\label{t3.28}
Let $E_0$ be a projection band in the vector lattice $E$, and $p_0$ the band projection from $E$ onto  $E_0$.
Then $p_0(\Phi_1(E))=\Phi_1(E)\cap E_0=\Phi_1(E_0)$.
If $E_1$ is another projection band in $E$ and $E=E_0\oplus E_1$ then $\Phi_1(E)=\Phi_1(E_0)\oplus \Phi_1(E_1)$.
\end{prop}
\par
Recall that an element $z$ in a vector lattice $E$ is said to be a {\it component} or a
{\it fragment} of $x$ if $z\bot(x-z)$, i.e. if $\bl z\br\wedge \bl x-z\br=0$.
The notations $x=y\sqcup z$ and $z\sqsubseteq x$ mean that $x=y+z$ with $y\bot z$ and that $z$ is a fragment
of $x$, respectively.
The set of all fragments of the element $x\in E$ is denoted by  $\cF_x$.
Let be $x\in E$. A collection $(\rho_\xi)_{\xi\in \Xi}$ of elements in $E$ is called a {\it partition of} $x$ if
$\bl\rho_\xi\br\wedge \bl\rho_\eta\br=0$, whenever $\xi\neq \eta$ and $x=\sum_{\xi\in \Xi}\rho_\xi$.
\begin{definition} \label{def:ddmjf1}
Let $E$ be a vector lattice and let $X$ be a real vector space.
An operator $T\colon E\rightarrow X$ is called \textit{orthogonally additive} if $T(x+y)=T(x)+T(y)$
whenever $x,y\in E$ are disjoint elements, i.e. if $\bl x\br\wedge \bl y\br =0$.
\end{definition}
It follows from the definition that $T(0)=0$.
It is immediate that the set of all orthogonally additive operators
is a real vector space with respect to the natural linear operations.
\par
So, the orthogonal additivity of an operator $T$ will be expressed as \linebreak $T(x\sqcup y)=T(x)+T(y)$.
\begin{definition}
Let $E$ and $F$ be vector lattices. An orthogonally additive operator $T\colon E\rightarrow F$ is called:
\begin{itemize}
  \item \textit{positive} if $Tx \geq 0$ holds in $F$ for all $x \in E$,
  \item \textit{order bounded} if $T$ maps any order bounded subset of $E$ into an order bounded subset of $F$.
\end{itemize}
An orthogonally additive order bounded operator $T\colon E\rightarrow F$ is called an \textit{abstract Uryson} operator.
\end{definition}
The set of all abstract Uryson operators from $E$ to $F$ we denote by $\mathcal{U}(E,F)$.
 If $F=\R$ then an element $f\in \mathcal{U}(E,\R)$ is called an {\it abstract Uryson functional}.
\par
A positive linear order bounded operator $A\colon E\to F$ defines a positive abstract Uryson operator by means
of $T(x)=A(\bl x\br)$ for each $x\in E$.
\par
 We will consider some examples.
The most famous ones are the nonlinear integral Uryson operators which are well known and thoroughly studied e.g.
in \cite{KZPS}, chapt.5.
\par
Let $(A,\Sigma,\mu)$ and $(B,\Xi,\nu)$ be $\sigma$-finite complete measure spaces and denote the completion of 
their product measure space by $(A\times B,\mu\times \nu)$.  
Let $K\colon A\times B\times\R\rightarrow\R$ be a function which satisfies the following conditions\footnote{$(C_{1})$ and $(C_{2})$ 
are called the Carath\'{e}odory conditions.}:
\begin{enumerate}
  \item[$(C_{0})$] $K(s,t,0)=0$ for $\mu\times\nu$-almost all $(s,t)\in  A\times B$;
  \item[$(C_{1})$] $K(\cdot,\cdot,r)$ is $\mu\times\nu$-measurable for all $r\in\R$;
  \item[$(C_{2})$] $K(s,t,\cdot)$ is continuous on $\R$ for $\mu\times\nu$-almost all $(s,t)\in A\times B$.
\end{enumerate}
Denote by $L_0(B,\Xi,\nu)$ or, shortly by $L_0(\nu)$, the ordered vector space of all $\nu$-measurable and 
$\nu$-almost everywhere finite functions on $B$ with the order
 $f\leq g$ defined as $f(t)\leq g(t) $ $\nu$-almost everywhere on $B$.
Then $L_0(\nu)$ is a Dedekind complete vector lattice.
Analogously, the space $L_0(A,\Sigma,\mu)$, or shortly $L_0(\mu)$, is defined. 
\par 
Given $f\in L_{0}(\nu)$ the function $|K(s,\cdot,f(\cdot))|$ is $\nu$-mea\-surable  for $\mu$-almost
all $s\in A$ and $h_{f}(s):=\int_{B}|K(s,t,f(t))|\,d\nu(t)$ is a well defined $\mu$-measurable function.
Since the function $h_{f}$ can be infinite on a set of positive measure, we define
\[
\text{Dom}_{B}(K):=\{f\in L_{0}(\nu)\colon \,h_{f}\in L_{0}(\mu)\}.
\]
\begin{example}[{\rm Uryson integral operator}]\label{Ex-0} 
Define an  operator 
\[   T \colon\text{Dom}_{B}(K)\rightarrow L_{0}(\mu)\]  
by 
\begin{equation}\label{U}
(Tf)(s)=\int_{B}K(s,t,f(t))\,\d\nu(t)\quad\mu\text{-a.e.}  
\end{equation}
Let $E$ and $F$ be order ideals in $L_{0}(\nu)$ and $L_{0}(\mu)$, respectively and $K$ a function
satisfying the conditions $(C_{0})$\,-\,$(C_{2})$. Then (\ref{U}) is an orthogonally additive, 
in general, not order bounded, integral operator acting from
$E$ to $F$ provided that $E\subseteq \text{Dom}_{B}(K)$ and $T(E)\subseteq F$. 
The operator $T$ is called  {\rm Uryson (integral) operator}.
\end{example}
\par
We consider the vector space $\R^m$ for $m \in \N$ as a vector lattice with the usual coordinate-wise order:
for any $x,y \in \R^m$ we set $x \leq y$ provided $e_i^*(x) \leq e_i^*(y)$ for all $i = 1, \ldots, m$,
where $(e_i^*)_{i=1}^m$ are the coordinate functionals on $\R^m$.
\begin{example} \label{Ex-1}
A map $T\colon \R^n\rightarrow\R^{m}$ belongs to $\mathcal{U}(\R^{n},\R^{m})$ if and only if there are real
functions $T_{i,j}:\R\rightarrow\R$,
$1\leq i\leq m$, $1\leq j\leq n$ satisfying the condition $T_{i,j}(0)=0$ and such that $T_{i,j}([a,b])$ is (order) bounded 
in $\R^m$ for each (order) interval $[a,b]\subset \R^n$, where the $i$-th component of the vector $T(x)$ is
calculated by the usual matrix rule, i.e.
$$
\left(T(x)\right)_i = e_i^*\bigl(T(x_{1},\dots,x_{n})\bigr) = \sum_{j=1}^{n}T_{i,j}(x_{j}), \quad i=1,\ldots,m
$$
In this case we write $T=(T_{i,j})$.
\end{example}
For more examples of abstract Uryson operators see \c{PP}.
\par\medskip
In $\mathcal{U}(E,F)$ the order is introduced as follows:  $S\leq T$ whenever $T-S$ is a positive operator.
Then $\mathcal{U}(E,F)$ becomes an ordered vector space.
If the vector lattice $F$ is Dedekind complete the following theorem is well known.
\begin{thm}[\cite{Maz-1}, Theorem~3.2.] \label{f}
Let $E$ and $F$ be vector lattices with $F$ Dedekind complete. Then $\mathcal{U}(E,F)$ is a Dedekind complete vector lattice. 
Moreover, for any $S,T\in \mathcal{U}(E,F)$ and $x\in E$ the following formulas hold
\begin{enumerate}
  \item $(T\vee S)(x)=\sup\{T(y)+S(z)\colon x = y \sqcup z\}$.
  \item $(T\wedge S)(x)=\inf\{T(y)+S(z)\colon x = y \sqcup z\}$.
  \item $T^{+}(x)=\sup\{Ty\colon y \, \sqsubseteq x\}$.
  \item $T^{-}(x)=-\inf\{Ty\colon y \, \sqsubseteq x\}$.
  \item $\bl T \br(x)=\left(T^+\vee T^-\right)(x) = \sup \{T(y)-T(z) \colon x=y\sqcup z\}$
  \item $|T(x)|\leq|T|(x)$.
  \end{enumerate}
\end{thm}
The formulas (1)\,-\,(5) are generalizations of the well known Riesz-Kanto\-rovich formulas for
linear regular operators (see \cite{Al}, Theorems 1.13 and 1.16).
\par
We also need the following result which represents the lattice operations in $\mathcal{U}(E,F)$ in terms of directed systems.

\begin{thm}[\cite{Maz-2}, Lemma~3.2] \label{fjjjjjgg}
Let $E$ and $F$ be vector lattices with $F$ Dedekind complete. Then for any $S,T\in\mathcal{U}(E,F)$ and $x\in E$ we have
\begin{enumerate}
  \item $\Big\{\sum\limits_{i=1}^{n}S(x_i)\wedge T(x_i)\colon x=\bigsqcup\limits_{i=1}^nx_i,\,n\in\N\Big\}\downarrow(S\wedge T)(x)$.
  \item $\Big\{\sum\limits_{i=1}^{n}S(x_{i})\vee T(x_i)\colon x=\bigsqcup\limits_{i=1}^nx_i,\,n\in\N\Big\}\uparrow(S\vee T)(x)$.
  \item $\Big\{\sum\limits_{i=1}^{n}|T(x_{i})|         \colon x=\bigsqcup\limits_{i=1}^nx_i,\,n\in\N\Big\}\uparrow|T|(x)$.
\end{enumerate}
\end{thm}
\par
\medskip
%
%
\section{Some properties of finite elements in the vector lattice $\cU(\R^n,\R^m)$}
%
%
In the one-dimensional case, by definition,
$\mathcal{U}(\R):=\cU(\R,\R)$ coincides with the set all functions $f\colon\R\to\R$,
such that $f(0)=0$ and for every (order) bounded set $A\subset\R$ its image $f(A)$ is also a bounded set.
\par
The vector lattice of abstract Uryson operators has a lot of finite elements meaning that $\Phi_1(\cU(\R^n,\R^m))\neq \{0\}$. 
The next proposition shows that $\{0\}\neq\Phi_1(\cU(\R))\neq\cU(\R)$ in the one-dimensional case.
For an arbitrary real function $f$ defined on $\R$ denote
 the set $\{x\in\R\colon f(x)\neq 0\}$ by $\text{supp}(f)$ and called it the {\it support} of $f$.
\begin{prop}\label{fin-1}
The set of all finite elements  $\Phi_{1}(\cU(\R))$ coincides with the set
\[
   \fF(\cU(\R))=\{f\in \cU(\R) \colon
    {\rm supp}(f)\subset[a,b],\; a,b\in\R \}.
\]
Moreover, $\Phi_2(\cU(\R))=\Phi_1(\cU(\R))$.
\end{prop}
\begin{proof}
Fix an arbitrary element $f\in\fF(\cU(\R))$.
Then $\text{supp}(f)\subset[a,b]$ for some $a,b\in\R$.
If $g\in\mathcal{U}(\R)$ then the number $c_g=\sup\{|g(x)|\colon x\in [a,b]\}$ belongs to $\R$.
Define the function
\[
z(x)=\begin{cases}
      1,\quad\text{if $x\in \text{supp}(f)$}     \\
      0,\quad\text{if $x\notin \text{supp}(f)$}.
\end{cases}
\]
Then $z$ is a bounded function on $\R$. Due to $z(0)=0$ and since in $\R$ any orthogonal representation
of some element $x$ is trivial, i.e. consists of $x$ and some zeros, the function $z$ belongs to $\mathcal{U}(\R)$.
For $x\in {\rm supp}(f)$ we have
\[
(|g|\wedge n|f|)(x)\leq|g|(x)\leq c_{g}z(x).
\]
The inequality trivially holds for $x\notin {\rm supp}(f)$. Therefore $f$ is a finite element
in $\cU(\R)$ and $z$ one of its majorants. Hence $\mathfrak{F}(\mathcal{U}(\R))\subset\Phi_{1}(\mathcal{U}(\R))$.
\par
In order to prove the converse assertion take an element $f\in\mathcal{U}(\R)$ such that there exist a
sequence $(x_{n})_{n=1}^{\infty}$ of real numbers $x_n\neq 0$ with the properties $\lim\limits_{n\to\infty}x_{n}=\infty$
and $f(x_{n})\neq 0$ for all $n\in\N$.
Let $f\in\Phi_{1}(\mathcal{U}(\R))$ and let $z$ be a majorant for $f$.
Then we may assume $z(x_{n})>1$ for all $n\in\N$.
By assumption, for every $g\in\mathcal{U}(\R)$ we have
\be\label{f3.1}
\sup\limits_{n}\{|g|\wedge n|f|\}\leq c_{g}\,z \quad \mbox{{\rm for some\;}} \,  c_{g}\in\R_{+},
\ee
where the existence of the supremum is guaranteed by the Dedekind completeness of $\cU(\R)$ (see Theorem \ref{f}).
In particular, for a function
$g\in \cU(\R)$ with
\[
g(x)= \begin{cases} \exp(z(x)),\;     \text{if $x=x_n \; \mbox{for\;} n=1,2,\ldots$}  \\
                    \quad 0, \qquad \quad \text{if $x\in \R,\, x\neq x_n$}                                                 \\
\end{cases}
\]
there exists  $n_{0}\in\N$, such that $g(x_{n})>c_{g}\,z(x_{n})$ for all $n\geq n_{0}$.
Fix $m\in\N$ such that $m|f(x_{n_{0}})|>g(x_{n_{0}})$.
Then
\[
\sup\limits_{n}\{|g|\wedge n|f|\}(x_{n_{0}})\geq(|g|\wedge m|f|)(x_{n_{0}})=
g(x_{n_{0}})> c_g\,z(x_{n_{0}}).
\]
This contradicts to (\ref{f3.1}) and therefore,
    $\fF(\cU(\R))\supset\Phi_{1}(\mathcal{U}(\R))$.
\par
Observe that the majorant $z\in \cU(\R)$ of $f$, constructed in the first part of the proof, 
is such that ${\rm supp}(z)={\rm supp}(f)$. So ${\rm supp}(z) \subset [a,b]$, and by what has
been proved one has
$z\in \Phi_1(\cU(\R))$. Therefore  $\fF(\cU(\R))\subset \Phi_2(\cU(\R)$.
Since $\Phi_2(\cU(\R))\subset \fF(\cU(\R))$ is clear from $\Phi_2(\cU(\R))\subset \Phi_1(\cU(\R))$,
we conclude that any finite element in $\cU(\R)$ is even totally finite.
\end{proof}
Guided by the previous proposition we describe the finite elements in the vector lattice $\cU(\R^n,\R^m)$
(see Example \ref{Ex-1}).
\begin{prop}\label{p2}
For the vector lattice $\cU(\R^n,\R^m)$ the ideal
$\Phi_1\left(\cU(\R^n,\R^m)\right)$
coincides with the set of all operators $T=(T_{i,j})\in\cU(\R^n,\R^m)$ for which each
function $T_{i,j} \;(i=1,\ldots, n;\, j=1,\ldots, m)$ satisfies the conditions $T_{i,j}(0)=0$ and
\be\label{f3.0}
   {\rm supp}(T_{i,j})\subset [a^{(ij)},b^{(ij)}] \; \mbox{ {\rm for some real
    interval} }\; [a^{(ij)},b^{(ij)}].
\ee
In this case $\Phi_1\left(\cU(\R^n,\R^m)\right) = \Phi_2\left(\cU(\R^n,\R^m)\right)$ also holds.
\end{prop}
\begin{proof}
The assertion of this proposition is established by a similar argument as for Proposition \ref{fin-1}.
Consider an operator $T\in \cU(\R^n,\R^m)$ such that its constituent functions $T_{i,j}$ have the properties
$T_{i,j}(0)=0$ and (\ref{f3.0}). Then for the interval $[a,b]$ with $a=\min_{ij}a^{(ij)}$ and $b=\max_{ij}b^{(ij)}$
one has  ${\rm supp}(T_{i,j})\subset [a,b]$ for all $i$ and $j$.
The set
\[
{\rm supp}(T)=\{x=(x_1,\ldots,x_n)\in \R^n\colon \exists i\in \{1,\ldots,m\} \; \mbox{ with }\;\sum_{j=1}^nT_{i,j}(x_j)\neq 0 \}
\]
will be called the {\it support} of the operator $T$. It is clear that
${\rm supp}(T)=\{x\in \R^n\colon T(x)\neq 0\}, \; 0\notin {\rm supp}(T)$ and $T(x)\neq 0$ for all
$x\in  {\rm supp}(T)$.
Define the map $Z=(Z_{i,j})\colon \R^n\to \R^m$ by
\[
Z(x)=\left(\begin{array}{ccc}
      Z_{1,1}(x_1) + & \cdots & +Z_{1,n}(x_n)\\
                     & \vdots                \\
      Z_{m,1}(x_1) + & \cdots & +Z_{m,n}(x_n)
     \end{array}\right),
\]
where $Z_{i,j}$ are real bounded functions with $Z_{i,j}(0)=0$ for all $i,j$. \\
Denote the set $\{1,2,\ldots, n\}$ by $N$.
For an arbitrary vector $w\in \R^n$ its support is the set
\[ {\rm supp}(w)=\{j\in N\colon w_j\neq 0\}. \]
In order to show the orthogonal additivity of $Z$ consider $x=u\sqcup v$ with $u$ and $v$ being fragments of $x$.
Then ${\rm supp}(u)\cap {\rm supp}(v)=\emptyset$ and
\[
x_j\,=\, \left\{\begin{array}{ccl}
      u_j,  & \quad & \text{if $j\in {\rm supp}(u)$  }   \\
      v_j,  & \quad & \text{if $j\notin {\rm supp}(u)$}.
\end{array}\right.
\]
Then
\[
Z(u\sqcup v)=\left(\begin{array}{l}
             \sum\limits_{j\in {\rm supp}(u)}Z_{1,j}(u_j) \\
                    \hspace{1.3cm}  \vdots                \\
      \sum\limits_{j\in {\rm supp}(u)}Z_{m,j}(u_j)
     \end{array}\right)\,+\,
              \left(\begin{array}{l}
              \sum\limits_{j\notin {\rm supp}(u)}Z_{1,j}(v_j) \\
                    \hspace{1.3cm}  \vdots                \\
              \sum\limits_{j\notin {\rm supp}(u)}Z_{m,j}(v_j)
     \end{array}\right).
\]
By using that $Z_{i,j}(u_j)=0$ for  $j\notin {\rm supp}(u)$ and
              $Z_{i,j}(v_j)=0$ for  $j\in {\rm supp}(u)$ (for all $i=1,\ldots,m$)
the summation in each of the coordinates of the last two vectors can be extended to the whole set $N$ and hence
one obtains $Z(u\sqcup v)=Z(u)+Z(v)$ and, so $Z\in \cU(\R^n,\R^m)$.

\par
For an arbitrary Uryson operator  $S=(S_{i,j})\in \cU(\R^n,\R^m)$ define the number
\[
c_S=\sup\big\{\sum_{i,j=1}^{m,n}|S_{i,j}(x_j)|\colon x=(x_1,x_2,\ldots,x_n)\in [a,b]\big\}.
\]
Then $c_S\in \R$ and for $x\in{\rm supp}(T)$ one has
\[
     (|S|\wedge n|T|)(x)\leq |S|(x)\leq c_S\,Z(x).
\]
For $x\notin {\rm supp}(T)$ the inequality is also true due to $T(x)=0$ in that case.
So $T$ is a finite element in $\cU(\R^n,\R^m)$ and $Z$ is one of its majorant.
\\
For the converse let $T\in \cU(\R^n,\R^m)$ be such that there is a sequence $(x^{(k)})$
of vectors $0\neq x^{(k)}\in \R^n$ with the properties that $T(x^{(k)})\neq 0$ and
$(x^{(k)})$ leaves any ball\footnote{\,  i.e. $\| x^{(k)}\| \to \infty$.}
in $\R^n$. If $T$ would belong to $\Phi_1(\cU(\R^n,\R^m))$ and $Z$ is a fixed majorant of $T$ then for
any operator $S\in \cU(\R^n,\R^m)$ one has
\begin{equation}\label{f2}
|S|\wedge nT\leq c_S\,Z \, \mbox{ for all }\, n\in \N \,\mbox{ and some number }\, c_S>0.
\end{equation}
In particular, this holds for an operator $0<S=(S_{i,j})\in \cU(\R^n,\R^m)$ with
\[
S_{i,j}(x_j)=\left\{\begin{array}{ccl}
                     \exp(Z_{i,j}(x_j)), & \mbox{ if } & x \in {\rm supp}(T) \\
                            0,           & \mbox{ if } & x\notin {\rm supp}(T)
                    \end{array}\right., \quad x=(x_1,\ldots,x_n).
\]
Then
\[ (S\wedge nT)(x^{(k)})=S(x^{(k)})=  \left(\begin{array}{c}
                                                \sum\limits_{j=1}^n\exp(Z_{1,j}(x^{(k)}_j))\\
                                                   \vdots \\
                                                 \sum\limits_{j=1}^n\exp(Z_{m,j}(x^{(k)}_j))
                                               \end{array}\right).
\]
It is clear that for sufficiently large $k$ the last vector is greater than $c_S\,Z(x^{(k)})$ what is in contradiction to
(\ref{f2}).
\end{proof}
\begin{rem} Actually it is  proved that
\[ T=(T_{i,j})\in \Phi_1(\cU(\R^n,\R^m))\; \mbox{ if and only if }\; T_{i,j}\in \Phi_1(\R) \]
for all  $i=1,\ldots,m;\, j=1,\ldots, n$.
\end{rem}
\par
\medskip
For a band\footnote{\, due to the Dedekind completeness of $F$ any band is a projection band.} 
$H\subset F$ we get a result for the abstract Uryson operators which is similar to Theorem 2 in \cite{Ch-3} for (linear)
regular operators.
\begin{prop}\label{P2}
Let $E,F$ be vector lattices with $F$ Dedekind complete and let $H$ be a band in $F$.
Then $\cU(E,H)$ is a projection band $\cU(E,F)$ and the following equation holds
\begin{eqnarray*}
 \Phi_{1}(\cU(E,H)) & = &\Phi_{1}(\cU(E,F))\cap\cU(E,H).
 \end{eqnarray*}
\end{prop}
\begin{proof}
Let $\pi\colon F\to H$ be the order projection in $F$ onto $H$.
It is clear that $\cU(E,H)$ is an order ideal in the $\cU(E,F)$.
Fix a net $(T_{\alpha})$ in $\cU_{+}(E,H)$, such that $T_{\alpha}\uparrow T$ for some $T\in \cU(E,F)$.
Then one  has  $T_{\alpha}=\pi T_{\alpha}\uparrow \pi
T\in\mathcal{U}_{+}(E,H)$.
Therefore  $T=\pi T$, i.e. the order ideal $\mathcal{U}(E,H)$
is a band and,  due to the Dedekind completeness of $\mathcal{U}(E,F)$, even a projection band.
Let  $\pi^\star\colon\cU(E,F)\rightarrow\cU(E,H)$ be the related order projection.
Then $\pi^{\star}(T)=\pi T$ holds for every  $T\in\cU(E,F)$. To finish the proof, refer to Theorem~2.11
from \cite{Ch-2}, saying that the finite elements in $\Phi_1(\cU(E,H))$ are exactly those finite elements of 
$\Phi_1(\cU(E,F))$, which belong to $\cU(E,H)$.
\end{proof}
\begin{rem} \label{r2}
By the mentioned result from \cite{Ch-2} there is proved even the equality
$  \pi^{\star}\left(\Phi_1(\cU(E,F)\right)= \Phi_1(\cU(E,H))$.
\end{rem}
\par
\medskip
%
%
\section{Finite elements in $\cU(E,\R)$}
%
%
\begin{definition}
A non-zero element $u$ of a vector lattice $E$ is called
an \textit{atom}, whenever $0 \leq x \leq |u|$, $0 \leq y \leq |u|$ and $x \wedge y = 0$
imply that either $x = 0$ or $y = 0$.
\end{definition}
If $u$ is an atom in $E$ then $\cF_u=\{0, u\}$. 
Note that a non-zero element $u$ of a vector lattice $E$ is called 
\textit{discrete}, if the ideal $I_u$ generated by $u$ in $E$ coincides with the vector subspace generated by $u$ in $E$, 
i.e. if $0\leq x< u$ implies $x=\lambda u$ for some $\lambda\in  \R_+$.  
We need the following properties of atoms.
\begin{prop}[\cite{LZ}, Theorem~26.4] \label{pr}
Let $E$ be  an Archimedean vector lattice. Then the following holds:
\begin{enumerate}
\item[(i)] Atoms and discrete elements are the same.  
\item[(ii)] For any atom $u$ the ideal $I_u$ is a projection band.
\item[(iii)] For any two atoms $u,v$ in $E$, either $u \bot v$, or $v = \lambda u$ for some $0 \neq \lambda \in \mathbb R$.
\end{enumerate}
\end{prop}
\begin{definition}
An Archimedean vector lattice $E$ is said to be {\it atomic}\footnote{\, or discrete.} 
if for each $0 < x\in E$ there is an atom $u\in E$ satisfying $0<u\leq x$. \\
A vector lattice is said to be {\it atomless} provided it has no atoms.
\end{definition}
Equivalently (see \c{AlBu}), $E$ is atomic, if and only if there is a collection $(u_i)_{i \in I}$ of atoms in $E$, 
such that $u_i \bot u_j$ for $i \neq j$ and for every $x \in E$
if $|x| \wedge u_i = 0$ for each $i \in I$ then $x = 0$. Such a collection is called a 
{\it generating disjoint collection of atoms}.  
\par
By Proposition~\ref{pr}, a generating collection of atoms in an  atomic vector lattice is unique,
up to a permutation and nonzero multiples. 
\par     
\smallskip
Let $E$ be a vector lattice. Consider any maximal collection of atoms $(u_i)_{i \in I}$ in $E$, the existence of which
is guaranteed by Proposition~\ref{pr} and by applying Zorn's Lemma. Let $E_0$ be the minimal band containing $u_i$ for all $i \in I$.
If $E_0$ is a projection band then $E = E_0 \oplus E_1$,
where $E_1 = E_0^d$ is the disjoint complement to $E_0$ in $E$, which is an atomless sublattice of $E$.
So, we obtain the following assertion.
\begin{prop} \label{pr-1}
Any vector lattice $E$ with the projection property\footnote{\, then $E$ is Archimedean.} has a decomposition into mutually complemented bands $E = E_0 \oplus E_1$, 
where $E_0$ is an atomic vector lattice and $E_1$ is an atomless vector lattice. 
\end{prop}
The following theorem is the first main results of this section and deals with finite elements
in atomic vector lattices.
\begin{thm}\label{op-3}
Let $E$ be an atomic vector lattice and $\phi\in\Phi_{1}(\cU(E,\R))$.
Then there exists only a finite set $\{e_{1},\dots, e_{n}\}$ of the mutually disjoint atoms in $E$, such
that $\varphi(e_{i})\neq 0$ for $i=1,\ldots,n$.
\end{thm}
\begin{proof}
If $E$ is a finite dimensional vector lattice then $E$ is isomorphic to $\R^{k}$ for some $k\in\N$
and  $\Phi_{1}(\cU(\R^k,\R))\neq \{0\}$ by Proposition~\ref{p2}. Then the coordinate vectors
\[  e^{(i)}=(0,\ldots,0,\overset{(i)}1,0,\ldots,0),\quad i=1,\ldots,k  \] 
are mutually disjoint atoms in $\R^k$.
Obviously, among them for each $0\neq \phi\in \Phi_1(\cU(\R^k,\R)$ there are some vectors,
on which the functional $\phi$ does not vanish.
\par
Let be $E$  an infinite-dimensional atomic vector lattice $E$. 
Let be  $\varphi\in \Phi_{1}(\cU(E,\R))$, $\varphi>0$ with a fixed positive majorant $\psi$.
Assume that for $\phi$ there exists an infinite set of mutually disjoint atoms $e_{n}\in E$, $n\in\N$ such
that $\varphi(e_{n})>0$ for every $n\in\N$.
Without restriction of generality\footnote{\, Otherwise replace $\phi$ by an element
with appropriate smaller values for $\phi(e_n)$.}\label{endlich} we may  assume
$\sum\limits_{n=1}^{\infty}\psi(e_{n})<\infty$.
For arbitrary $T\in\cU_+(E,\R)$  there exists a number $c_{T}>0$ such that
$ (T\wedge n\varphi)x\leq c_T\psi(x)\; \mbox{ for every }\; n\in\N \mbox{ and } x\in E, $
what implies $(\pi_{\varphi}T)x\leq c_{T}\psi(x)$.
By applying the formula
\be\label{f48}
(\pi_{\varphi}T)x=\sup\limits_{\varepsilon>0}\,
           \inf\limits_{y\in\cF_x}\{Ty\colon\varphi(x-y)\leq\varepsilon \varphi(x)\}.
\ee
(which was proved for any $x\in E$ in \cite{PW}, Formula (3.8))
to the atom $e_n$ and by taking into account that,  due to $\phi(0)=0$ and $\cF_{e_n} =\{0, e_n\}$,
the element $y=e_n$ is the only feasible in formula (\ref{f48}) (applied to $e_n$) we get
\[
(\pi_{\varphi}T)e_{n}=
\sup\limits_{\varepsilon>0}\, \inf\limits_{y\in\cF_{e_n}}\{Ty\colon\varphi(e_{n}-y)\leq\varepsilon
                     \varphi(e_{n})\} = Te_{n}.
\]
Therefore
 \be\label{f5.1}
 Te_{n}\leq c_{T}\psi(e_{n})\; \mbox{ for every }\; n\in\N
 \ee
 and so, $\sum\limits_{n=1}^{\infty}Te_{n}<\infty$ for each $T\in\cU(E,\R)$.
For every $n\in \N$ choose a natural number $k_n\in\N$ such that $\psi(e_{k_{n}})<\frac{1}{(n+1)^{4}}$ and
define numbers $\beta_{k_n}$ satisfying the condition $\frac{1}{(n+1)^{3}}<\beta_{k_n} <\frac{1}{(n+1)^{2}}$.
Take now a functional $T\in\cU_{+}(E,\R)$ such that
\[
Te_{k}=\left\{\begin{array}{ccl}
                   \beta_{k_n}, & \mbox{ if } & k=k_n   \\
                   \psi(e_{k}), & \mbox{ if } & k\neq k_n
                    \end{array}\right. \quad \mbox{for\;} \; k=1,2,\ldots \,.
\]
It is clear that $\sum\limits_{k=1}^{\infty}Te_{k}<\infty$.
Fix $n_{0}\in\N$ with $c_{T}<n_{0}$, where $c_T$ is the constant number for the functional $T$ one has
according to the finiteness of $\phi$. Then
\[
c_{T}\psi(e_{k_{n_{0}}})<n_{0}\psi(e_{k_{n_{0}}})<\frac{1}{(n_{0}+1)^{3}}<\beta_{k_{n_{0}}}=Te_{k_{n_{0}}}.
\]
This is a contradiction to (\ref{f5.1}).
\end{proof}
Our aim now is to establish that for an atomless vector lattice the band of $\sigma$-laterally continuous
abstract Uryson functionals possesses only the trivial finite element.
\begin{definition}
A sequence $(x_{n})_{n\in\N}$ in a vector lattice $E$ is said
to be {\it laterally converging} to $x \in E$ if $x_{n} \sqsubseteq x_m \sqsubseteq x$ for all $n<m$
and $x_n \overset{\rm (o)}\longrightarrow x$.
In this case we write 
$x_n \overset{\rm lat}\longrightarrow x$.
For positive elements $x_n$ and $x$ the notion $x_n \overset{\rm lat}\longrightarrow x$
means that $x_n\in \cF_x$, 
$x_n\uparrow x$ and $x_n \overset{\rm lat}\longrightarrow x$.
\end{definition}
\begin{definition}
Let $E, F$ be  vector lattices. An orthogonally additive operator $T\colon E\to F$ is called 
{\it $\sigma$-laterally continuous} if $x_n \overset{\rm lat}\longrightarrow x$ 
implies 
$Tx_n \overset{\rm (o)}\longrightarrow Tx$.
The vector space  of all $\sigma$-laterally continuous
abstract Uryson operators from $E$ to $F$ is denoted by 
$\mathcal{U}_{\sigma c}(E,F)$.
\end{definition}
It turns out that $\mathcal{U}_{\sigma c}(E,F)$ is
a projection bands in $\mathcal{U}(E,F)$ (\cite{Maz-1}, Proposition 3.8).
We need the following auxiliary lemma.
\begin{lemma}{\label{op-4001}}
Let $E$ be an atomless vector lattice, $\varphi\in\cU_{\sigma c}(E,\R)$  and $\phi(x)>0$ for some vector
$x\in E$. Then there exists a sequence  $(x_{n})_{n\in\N}$ of mutually disjoint
fragments of $x$, such that $\phi(x_{n})>0$, for every $n\in\N$.
\end{lemma}
\begin{proof}
Assume that for every fragments $x', x''$ of $x$ with $x'\bot x''$, we have $\phi(x')=0$ and
$\phi(x'')=\phi(x)$. 
Put\footnote{\, Since $E$ is atomless there are nontrivial (i.e. different from $0$ and $x$) elements in $\cF_x$.}   
$x_1=x'$ and
consider in the next step the element $x''$.
By repeating the procedure we construct a sequence $(x_n)_{n=1}^{\infty}$ of mutually disjoint
fragments of $x$ such that
$x=\bigsqcup\limits_{n=1}^{\infty}x_{n}$ and $u_{n}=\bigsqcup\limits_{i=1}^{n}x_{i}\overset{\rm lat}\longrightarrow x$.
However $\phi(u_{n})=0$ for each $n\in \N$, what
is a contradiction to the fact that $\varphi$ belongs to $\in\cU_{\sigma c}(E,\R)$.
\end{proof}
Now we deal with lateral ideals in vector lattices.
\begin{definition}\label{def:adm}
A subset $D$ of a vector lattice $E$ is called a \it{lateral ideal} if the following
conditions\footnote{\,that means, $D$ is saturated in the sense of (i) and (ii).} hold:
\begin{enumerate}
\item[(i)] if $x\in D$ then $y\in D$ for every $y\in\mathcal{F}_{x}$,
\item[(ii)] if $x,y\in D$, $x\bot y$ then $x+y\in D$.
\end{enumerate}
\end{definition}
\begin{example}\label{adm-1}
Let $E$ be a vector lattice. Every order ideal in $E$ is a lateral ideal.
\end{example}
\begin{example}\label{adm-2}
Let $E$ be a vector lattice and $x\in E$. Then $\cF_x$  is a lateral ideal (see \c{BAP}, Lemma 3.5).
\end{example}
\par
\begin{lemma}\label{Gum-2}
Let be $E$ be a vector lattice and $\mathcal{D}=(D_n)_{n\in\N}$ a sequence of mutually
disjoint lateral ideals in $E$. Then the set
\[
L(\mathcal{D}):=\Big\{\bigsqcup\limits_{i=1}^{k}x_{i}\colon x_{i}\in D_{n_{i}},\,1\leq i\leq k,\,k\in\N\Big\}
\]
is also a lateral ideal.  
\end{lemma}
\begin{proof}
Take arbitrary elements $x,y\in L(\mathcal{D})$, such that $x\bot y$. 
Then
\begin{eqnarray*}
     x & = & \bigsqcup\limits_{i=1}^{k}x_{i}  \;\mbox{ for }\; x_{i}\in D_{n_{i}}, \qquad \;
\,y=\bigsqcup\limits_{j=1}^m y_{j}\;\mbox{ for } \;y_{j}\in D_{n_j},\\
 x_{i} & \bot & y_{j}  \;\mbox{ for }\; 1\leq i\leq k,\,1\leq j\leq m\quad\mbox{and} \\
 x+y   & = &  \bigsqcup\limits_{r=1}^{k+m}z_{r}, \;\mbox{where}\;
          z_{r}= \left\{ \begin{array}{ccl}
                         x_{r},  & \mbox{if } & 1\leq r\leq k  \\
                         y_{r-k},& \mbox{if } & k  <  r\leq k+m.
                         \end{array}\right.
\end{eqnarray*}
Hence the condition (ii) from Definition~\ref{def:adm} is proved for $L(\mathcal{D})$.
Now, let $x\in L(\mathcal{D})$ and $y\in\mathcal{F}_{x}$.
Then $x=\bigsqcup\limits_{i=1}^{k}x_{i}$ with $x_{i}\in D_{n_{i}}$.
By the Riesz decomposition property every $x_{i}$ has a decomposition
$x_{i}=y_{i}\sqcup z_{i}$, where $y=\bigsqcup\limits_{i=1}^{k}y_{i}$ and
$y_{i}$ belongs to $D_{n_{i}}$ due to $y_i\sqsubseteq x_i\in D_{n_{i}}$ for $i=1,\ldots,k$. 
So, the condition (i) from Definition~\ref{def:adm}
is also shown.
\end{proof}
\par
\smallskip
The following extension property of positive orthogonally additive operators was proved in \c{Gum}.
\begin{thm}[\cite{Gum}, Theorem 1]\label{Gum}
Let $E,F$ be  vector lattices with $F$ Dedekind complete and $D$ a lateral ideal in $E$.
Let $T\colon D\to F$ be a positive orthogonally additive operator such that the set
$T(D)$ is order bounded in $F$.
Then there exists an operator $\widetilde{T}_{D}\in\mathcal{U}_{+}(E,F)$ with
$Tx=\widetilde{T}_{D}x$ for every $x\in D$.
\end{thm}
The operator $\widetilde{T}_{D}\colon E\to F$ (or, for simplicity,
$\widetilde{T}\colon E\to F$)
is  defined by the formula\footnote{\,At least $0\in D\cap\cF_x$ for any $x\in E$.}
\be\label{f50}
    \widetilde{T}x=\sup\{Ty\colon y\in\mathcal{F}_{x}\cap D\}.
\ee
\par
Such an extension of $T$ is not unique. 
Due to the next lemma the operator $\widetilde{T}\in\mathcal{U}_{+}(E,F)$ is called the {\it minimal extension}
(with respect to $D$) of the positive, order bounded orthogonally additive
operator $T \colon D\to F$.
\begin{lemma}\label{Gum-1}
Let $E,F,D,T,\widetilde{T}$ be  as in Theorem~\ref{Gum} and let $R\colon E\to F$ be a positive
abstract Uryson operator such that $Rx=Tx$ for every $x\in D$.
Then  $\widetilde{T}x\leq Rx$ for every $x\in E$.
\end{lemma}
\begin{proof}
Take an arbitrary element $x\in E$ and $y\in\mathcal{F}_{x}\cap  D$. Then
\begin{eqnarray*}
R(x) &   =  & R(x-y)+R(y)=R(x-y)+Ty\geq Ty \quad\mbox{and} \\
R(x) & \geq & \sup\{Ty\colon\,y\in\mathcal{F}_{x}\cap D\}=\widetilde{T}x.
\end{eqnarray*}
\end{proof}
\par
Now the second main result of this section can be provided.
\begin{thm}\label{op-4}
Let $E$ be an atomless vector lattice. Then  $\Phi_{1}(\cU_{\sigma c}(E,\R))=\{0\}$.
\end{thm}
\begin{proof}
Assume that there exists an element $\phi\in\Phi_{1}(\cU_{\sigma c}(E,\R)), \,\phi>0$. Fix a positive 
laterally $\sigma$-continuous majorant $\psi$ for $\phi$.
Then for some $x\in E, \,x\neq 0$ one has $\phi(x)>0$. Since $E$ is atomless by Lemma~\ref{op-4001} it can be deduced  
that there exists a sequence $(x_{n})_{n\in\N}$ of mutually disjoint fragments of $x$ such that $\phi(x_{n})>0$   
for every $n\in\N$. 
Take now a positive functional $T\in \cU_{\sigma c}(E,\R)$ with $T(x)>0$ and $T(x_n)>0$ 
for every $n\in \N$ (e.g. $T=\phi$). 
Since $\phi$ is a finite element there is some $c_T>0$, such that 
$(\pi_\varphi T)(x_{n})\leq c_T\psi(x_n), \, n\in\N$.  
Consider the functional  
\[    G_n\colon \cF_{x_n}\to \R_+       \] 
defined on the lateral ideal $\cF_{x_n}$
by $G_n(y)=(\pi_\phi T)(y)$.
Then $G_n$ is an orthogonally additive functional, the set $G_n(\cF_{x_n})$ is (order) bounded
and $G_n(x_n)=(\pi_\phi T)(x_n), \, n\in \N$.
According to Theorem~\ref{Gum}, $G_n$ can be extended to the functional 
$\widetilde{G}_{n}\in\cU_+(E,\R)$ which, according to (\ref{f50}), 
is well defined on $E$ for every $n\in \N$.
Since $(\pi_\phi T)(x_n)\geq (T\wedge n\phi)(x_n)$ one has $(\pi_\phi T)(x_n)>0, \,n\in \N$. 
By Lemma~\ref{Gum-1} the inequality $\widetilde{G}_{n}(x)\leq (\pi_\phi T)(x)$ holds for every $x\in E$, i.e.
$\widetilde{G}_n\leq\pi_\phi T$ and $\widetilde{G}_n\in\{\phi\}^{\bot\bot}, \, n\in\N$.
Moreover, $\widetilde{G}_{n}(x_n)=(\pi_\phi T)(x_n)>0$.
It is clear that  $\widetilde{G}_n\leq c_T\psi$, $n\in\N$.
In view of the fact\footnote{\,see footnote at page \pageref{endlich} .} 
that $\sum\limits_{n=1}^{\infty}\psi(x_n)= \psi(x)<\infty$, for every $k\in\N$
there exists an index $n_k$, such that $\psi(x_{n_k})<\frac{1}{k^4}$.
For every $k\in \N$ fix now numbers $\beta_{n_k}$ such that
\[
\frac{1}{k^{3}\widetilde{G}_{n_k}(x_{n_k})}<\beta_{n_k}<\frac{1}{k^{2}\widetilde{G}_{n_k}
(x_{n_k})}.
\]
Then
\[
\sum_{k=1}^{\infty}\beta_{n_k}\widetilde{G}_{n_k}(x_{n_k})<
\sum_{k=1}^{\infty}\frac{\widetilde{G}_{n_k}(x_{n_k})}{k^{2}\widetilde{G}_{n_k}
(x_{n_k})}=\sum_{k=1}^{\infty}\frac{1}{k^{2}}<\infty.
\]
Observe that $\mathcal{F}=(\mathcal{F}_{x_{n_{k}}})_{_{k\in \N}}$
is a sequence of mutually disjoint lateral ideals. Thus by Lemma~\ref{Gum-2} the set $L(\mathcal{F})$ 
is also a lateral ideal.
Denote the operators $G_{n_{k_i}}$ by $R_i$
and define the operator $R\colon L(\mathcal{F})\to\R_+$ by the formula
\[
R\Big(\bigsqcup\limits_{i=1}^{k}u_{i}\Big)=\sum\limits_{i=1}^{k}R_i(u_i).
\]
It will be shown that $R$ is an orthogonally additive operator from the lateral ideal $L(\mathcal{F})$ to
$\R$. Indeed, take $u,v\in L(\mathcal{F})$ such that $u\bot v$. Then
\begin{eqnarray*}
R(u+v) & = & R\Big(\bigsqcup\limits_{i=1}^{k}u_{i}+\bigsqcup\limits_{j=1}^{m}v_{j} \Big)=
R\Big(\bigsqcup\limits_{r=1}^{k+m}z_{r}\Big)=\sum\limits_{r=1}^{k+m}R_r(z_r) \\
      & = & \sum\limits_{i=1}^{k}R_i(u_i)+\sum\limits_{j=1}^{m}R_j(v_j)=R(u)+R(v), \\
      &     \mbox{where }\;&
          z_{r}= \left\{ \begin{array}{ccl}
                         u_{r},  & \mbox{if } & 1\leq r\leq k  \\
                         v_{r-k},& \mbox{if } & k  <  r\leq k+m.
                         \end{array}\right.
\end{eqnarray*}
For any element $u=\bigsqcup\limits_{i=1}^k u_i\in L(\mathcal{F})$ with $u_i\in \cF_{x_{n_{k_i}}}$, due to
\[
 G_{k_i}(u_i)=(\pi_\phi T)(u_i)\leq c_T\psi(u_i)\leq c_T\psi(x_{n_{k_i}})<\frac{c_T}{(k_i)^4},
\]
one has
\[
R(u)=R\Big(\bigsqcup\limits_{i=1}^{k}u_{i}\Big)=\sum\limits_{i=1}^{k}R_{i}(u_i)<
   c_T \sum\limits_{i=1}^{k}\frac{1}{(k_i)^4}<c_T\sum_{k=1}^{\infty}\frac{1}{k^{4}},
\]
and therefore, the operator $R$ is order bounded.
In view of Theorem \ref{Gum} there exists the minimal extension $\widetilde{R}$ of $R$, 
which is a positive abstract Uryson functional $\widetilde{R}\colon E\to\R$ such that
$\widetilde{R}(x)=\sup\{R(v)\colon v\in\mathcal{F}_{x}\cap\, L(\mathcal{F})\}$ for any $x\in E$.
Observe that $\widetilde{R}(x_{n_{k}})=\widetilde{G}_{n_{k}}(x_{n_{k}})$ for every $k\in \N$.
Let be $S\colon E\to \R$ an abstract, positive Uryson functional such that $S\geq\widetilde{G}_{n_{k}}$
for any $k\in \N$ and fix an arbitrary element $x\in E$.
Then for every decomposition $x=y+z$, where $y\bot z$ and
\[
z\in L(\mathcal{F}),\;\mbox{ i.e. }\; z=\bigsqcup\limits_{i=1}^{m}u_{i}\;\mbox{ with }\;
u_{i}\in\mathcal{F}_{x_{n_{k_{i}}}}\,\mbox{ for some } m,
\]
one has
\[S(x)=S(y+z)\geq S(z)=S\Big(\bigsqcup\limits_{i=1}^{m}u_{i}\Big)
\geq \sum\limits_{i=1}^{m}G_{n_{k_i}}(u_i)= \sum\limits_{i=1}^{m}R_i(u_i).
\]
Passing to the supremum over all fragments $z\in L(\mathcal{F})$ we conclude that $S(x)\geq\widetilde{R}(x)$
for every $x\in E$. Hence $\widetilde{R}=\sup\limits_{i\in\N}\{\widetilde{G}_{n_{k_i}}\}$ in $\mathcal{U}(E,F)$.
Using the fact that $\widetilde{G}_{n_{k_i}}\in\{\varphi\}^{\bot\bot}$
for every $i\in\N$, we deduce that $\widetilde{R}\in\{\varphi\}^{\bot\bot}$.
Therefore a number $c_{\widetilde{R}}>0$ exists with $\widetilde{R}\leq c_{\widetilde{R}}\psi$.
For any number $k\in\N$ such that $c_{\widetilde{R}}\leq k$ one has
\[
c_{\widetilde{R}}\psi(x_{n_k}) \leq k\psi(x_{n_k}) < \frac{1}{k^{3}} < \beta_{n_k}G_{n_k}(x_{n_k})=
\beta_{n_k}\widetilde{G}_{n_k}(x_{n_k})=\widetilde{R}(x_{n_k}),
\]
what is a contradiction.
\end{proof}
Now we are ready to put together the Theorems~\ref{op-3} and \ref{op-4}.
\begin{thm}\label{op-5}
Let $E$ be a vector lattice with the  projection property and
$\phi\in\Phi_{1}(\mathcal{U}_{\sigma c}(E,\R))$.
Then there exists a finite dimensional projection band $M$ generated by a (finite) number of mutually 
disjoint  atoms in  $E$, such that $\varphi(x)=0$ for every $x\in M^{\bot}$.
\end{thm}
\begin{proof}
By Proposition~\ref{pr-1} there exists a decomposition into mutually complemented
bands $E=E_{0}\oplus E_1$, where $E_0$
is an atomic vector lattice and $E_1$ is an atomless vector lattice.
For the finite elements in $\cU_{\sigma c}(E, \R)$ there holds the equality 
\[
\Phi_1(\cU_{\sigma c}(E_0\oplus E_1, \R))= \Phi_1(\cU_{\sigma c}(E_0, \R))\oplus\Phi_1(\cU_{\sigma c}(E_1, \R)).
\]
For that it is proved first that
\begin{align}\label{a}
\cU_{\sigma c}(E_0\oplus E_1, \R)= \cU_{\sigma c}(E_0, \R)\oplus \cU_{\sigma c}(E_1, \R).
\end{align}
Take $f_{i}\in\cU_{\sigma c}(E_i, \R), \, i=0,1$. 
Define the functional $f=f_{0}\oplus f_{1}$ for each $x=(x_0,x_1)\in E$ by the formula
$f(x_{0},x_{1})=f_{0}(x_{0})+f_{1}(x_{1})$, where $x_0\in E_0, \, x_1\in E_1$. 
The functional $f$ belongs to the set $\cU_{\sigma c}(E_0\oplus E_1, \R)$. 
Indeed, take a sequence $(x_n)_{n\in \N}$ in $E_0\oplus E_1$, 
\hide{net $(x_{\alpha})_{\alpha\in\Lambda}$ in $E_0\oplus E_1$, 
such that $x_\alpha \overset{\rm lat}\longrightarrow x$, where $x_\alpha=(x_{\alpha 0},x_{\alpha 1})$ and $x=(x_0,x_1)$. 
}
such that $x_n \overset{\rm lat}\longrightarrow x$, where $x_n=(x_{n 0},x_{n 1})$ and $x=(x_0,x_1)$. 
Then 
\begin{align*}
\hide{f(x_{\alpha})=f(x_{\alpha 0},x_{\alpha 1})=f_{0}(x_{\alpha 0})+
f_{1}(x_{\alpha 1})
}
f(x_n)=f(x_{n 0},x_{n 1})=f_{0}(x_{n 0})+f_{1}(x_{n 1})\overset{\rm (o)}\longrightarrow f_0(x_0)+f_1(x_1)=f(x).
\end{align*}
On the other hand, let $f\in\cU_{\sigma c}(E_0\oplus E_1, \R)$. 
Denote by $f_{i}$ the restriction of $f$ on $E_{i}, \, i=0,1$. Then $f=f_{0}+f_{1}$, 
with $f_{i}\in\cU_{\sigma c}(E_{i}, \R)$. 
\par
Now it will be shown that $\cU_{\sigma c}(E_0, \R)^{\bot}=\cU_{\sigma c}(E_1, \R)$ and, 
therefore \linebreak $\cU_{\sigma c}(E_0, \R)$ and $\cU_{\sigma c}(E_1, \R)$ are mutually disjoint bands 
in $\cU_{\sigma c}(E_0\oplus E_1, \R)$. Hence the equality (\ref{a}) will be established. 
Let $0\leq f_{i}\in \cU_{\sigma c}(E_{i}, \R), \, i=0,1$ and $x\in E$. 
Then $x=x_0\sqcup x_1$ with $x_i\in E_i,\, i=0,1$ and 
\[
(f_{0}\wedge f_{1})(x)=\inf\{f_{0}(y)+f_{1}(z)\colon x=y\sqcup z\}\leq f_{0}(x_1)+f_{1}(x_0)=0.
\]
Since  $\cU_{\sigma c}(E_0\oplus E_1, \R)$ is Dedekind complete,
Proposition \ref{t3.28} guarantees the required equality
$$
\Phi_1(\cU_{\sigma c}(E_0\oplus E_1,\R))= \Phi_1(\cU_{\sigma c}(E_0,\R))\oplus \Phi_1(\cU_{\sigma c}(E_1,\R)).
$$
\par
For $\phi\in \Phi_1(\cU_{\sigma c}(E,\R))$ one has now $\phi=\phi_0+\phi_1$, where $\phi_i\in \Phi_1(\cU_{\sigma c}(E_i,\R))$ 
for $i=0,1$. 
Theorem~\ref{op-4} implies $\phi_{1}=0$ and, by Theorem~\ref{op-3} there exist only finite 
many $e_{1},\dots,e_{n}$ of mutually disjoint atoms in $E_{0}$, such that $\phi_{0}(e_{i})\neq 0,\, i=1,\ldots,n$.
Denote by $M$ the band in $E_{0}$, generated by $e_{1},\dots,e_{n}$. 
In view of the assumption on $E$ the band $M$ is a projection band in $E_0$. 
Then every element $x\in M^{\bot}$ is a linear combination of atoms disjoint to $M$ and therefore,
$\phi_{0}(x)=0$. Thus $\phi(x)=0$ for all $x\in M^\bot$. 
\end{proof}
\par\medskip
%
\section{Rank one operators as finite elements in $\cU(E,F)$}
%
%
Let $E,F$ be  vector lattices. An operator $T\in\cU(E,F)$ is called  a {\it finite rank} operator, if
$T=\sum\limits_{i=1}^{n}\phi_{i}\otimes u_{i}$ for some $n\in \N$, where $\phi_{i}\in
\cU(E,\R),\,u_{i}\in F$ and,  $(\phi_i\otimes u_i)(x)=\phi_i(x)\,u_i, \,x\in E$
for all $i=1,\ldots,n$. Similarly to the case of linear rank one operators in the vector lattice of
regular operators the modulus of a rank one abstract Uryson operator has a simple structure.
\begin{prop}\label{mod}
Let $E,F$ be vector lattices, with $F$ Dedekind complete. Then the modulus of the operator
$T=\phi\otimes u\in \cU(E,F)$ is the operator $|T|=|\phi|\otimes|u|$.
\end{prop}
\begin{proof}
Using the relation (3) of Theorem~\ref{fjjjjjgg} one has 
\begin{eqnarray*}
|T|(x) & = &|\phi\otimes u|(x) =  \sup\Big\{\sum_{i=1}^{n}|(\phi\otimes u)(x_i)|\colon
                             x =  \bigsqcup\limits_{i=1}^{n}x_i,\, n\in\N  \Big\}  \\
       & = & \sup\Big\{\sum_{i=1}^{n}|\phi(x_i)u|\colon
                             x = \bigsqcup\limits_{i=1}^nx_i,\,n\in\N   \Big\} \\
       & = & |u|\sup\Big\{\sum_{i=1}^{n}|\phi(x_i)|\colon
                             x = \bigsqcup\limits_{i=1}^{n}x_i,\,n\in\N\Big\}
                               =  |\phi|(x)|u|.
\end{eqnarray*}
\end{proof}
The following theorem tells us that the constituent parts of an abstract Uryson rank one operator $T$
are finite elements in the corresponding vector lattices, whenever $T$ is a finite element in $\cU(E,F)$.
Recall that the order dual of the vector lattice $F$ is denoted by $F^{\sim}$.
The expression $F^{\sim}$ {\it separates the points} of $F$ means that for each $0\neq y\in F$
there exists some $f\in F^{\sim}$ with $f(y)\neq 0$. The order dual separates the points of $F$, e.g., if $F$ is a 
Dedekind complete Banach lattice.
\begin{thm}
Let $E,F$ be vector lattices  with $F$ Dedekind complete and $F^{\sim}$ separates the 
points of $F$.
Let  $T\in\cU(E,F)$ be a rank one operator, i.e.  $T=\phi\otimes u$ for some $\phi\in \cU(E,\R)$ and $u\in F$.
If $T\in\Phi_{1}(\cU(E,F))$ then  $\phi\in\Phi_{1}(\cU(E,\R))$
and $u\in\Phi_{1}(F)$.
\end{thm}
\begin{proof}
By Proposition~\ref{mod} it suffices to consider a positive abstract Uryson operator $T=\phi\otimes u$,
with  $0<\phi\in\cU(E,\R)$, $0<u\in F$.
By assumption $T$ is a finite element and therefore an operator
$Z\in\cU_+(E,F)$ exists, such that for every $S\in\cU(E,F)$ the inequality
\[             |S|\wedge nT\leq c_{S}Z         \]
holds for some $c_S>0$ and  every $n\in\N$.
Consider the operator $S=\phi\otimes h$ for $h\in F$ and fix $m\in\N$.
Then for every  $x\in E$ we have
\begin{eqnarray*}
      c_{S}Z(x) & \geq & (|S|\wedge mT)(x) \\
                &  =  & \inf\Big\{\sum_{i=1}^{n}|S|(x_{i})\wedge
                      mT(x_{i})\colon x=\bigsqcup_{i=1}^{n}x_{i}, \,n\in\N\Big\}  \\
                &  =  & \inf\Big\{\sum_{i=1}^{n}|h|\phi(x_{i})\wedge m u\,\phi(x_{i})\colon
                      x=\bigsqcup_{i=1}^{n}x_{i},\; n\in\N\Big\} \\
                &  =  & \inf\Big\{\sum_{i=1}^{n}\phi(x_{i})(|h|\wedge (m u))\colon
                      x=\bigsqcup_{i=1}^{n}x_{i},\;n\in\N\Big\} \\
                &  =  & \phi(x)(|h|\wedge (m u)).
\end{eqnarray*}
The abstract Uryson functional $\phi$ is a nonzero positive element in $\cU(E,\R)$,
hence there exists $x_{0}\in E$, such that $\phi(x_{0})>0$. By means of the last estimation
the inequality
\[
|h|\wedge(m u)\leq \frac{c_{S}}{\phi(x_0)}Z(x_{0})=\mu Z(x_{0})
\]
holds with $\mu=\frac{c_{S}}{\phi(x_0)}\in\R_+$ and arbitrary $m\in\N$.
Since $h$ is an arbitrary element of $F$ it is proved that $u\in\Phi_{1}(F)$.
\par
For proving the second assertion consider the rank one operator $S=\theta\otimes u$,
where $\theta\in\cU(E,\R)$ is arbitrary. For every $x\in E$ and $n\in\N$ we may write
\[
(|\theta|\wedge
n\phi)(x)u\leq(|\theta|(x)u)\wedge(n\phi(x)u)=(|S|(x))\wedge (nT(x)).
\]
Then for every disjoint partition $\{x_{1},\ldots,x_{n}\}$
of $x$, i.e. $x=\bigsqcup\limits_{i=1}^n x_i$, by using the orthogonal additivity of the
abstract Uryson functional $|\theta|\wedge n\phi$ we have
\[
(|\theta|\wedge n\phi)(x)u=\sum\limits_{i=1}^{n}(|\theta|\wedge n\phi)(x_{i})u\leq
        \sum\limits_{i=1}^{n}|S|(x_{i})\wedge nT(x_{i}).
\]
After taking the infimum over all disjoint partitions of $x$ on the right side of last formula, one has 
\[
(|\theta|\wedge n\phi)(x)u\leq(|S|\wedge nT)(x)\leq c_{S}Z(x)
\]
for every  $x\in E$ and $n\in\N$.
If now $\tau$ is a positive linear functional on $F$, such that $\tau(u)=1$  then
\[
\tau\big((|\theta|\wedge n\phi)(x)u\big)=(|\theta|\wedge n\phi)(x)\leq\tau(c_{S}Z(x))=
                                   c_{S}(\tau\circ \,Z)(x),
\]
where $x\in E,\,n\in\N$, and  $\tau\, Z \in\cU_+(E,\R)$.
Thus  $\phi\in\Phi_{1}(\cU(E,\R))$ is proved and, $\tau\circ \,Z $ is one of its majorants.
\end{proof}
\hspace*{-3mm} The converse statement still remains to be an open question. 

\end{document}